\documentclass[a4paper,12pt]{article}
\usepackage{amssymb,amsmath,amsthm,latexsym}
\usepackage{amsfonts}
\usepackage{amsfonts}
\usepackage{graphicx}
\usepackage[pdftex,bookmarks,colorlinks=false]{hyperref}
\usepackage{verbatim}
\usepackage[font=small,labelfont=bf]{caption}
\usepackage{subcaption}
\newtheorem{theorem}{Theorem}[section]

\newtheorem{corollary}[theorem] {Corollary}
\newtheorem{definition}[theorem]{Definition}

\newtheorem{proposition}[theorem]{Proposition}

\title{\bf The Sum and Product of Chromatic Numbers of Graphs and their Line Graphs}

\author{{\bf Sunny Joseph Kalayathankal\footnote{Department of Mathematics, Kuriakose Elias College, Mannanam, Kottayam - 686561, Kerala, email:{\em sunnyjoseph2014yahoo.com}}}       ~and {\bf Susanth C \footnote{Department of Mathematics, Vidya Academy of Science \& Technology, Thalakkottukara, Thrissur - 680501, email: {\em susanth\_c@yahoo.com}}}}
\date{}

\begin{document}
\maketitle

\begin{abstract}
The bounds on the sum and product of chromatic numbers of a graph and its complement are known as Nordhaus-Gaddum inequalities. In this paper, some variations on this result is studies. First, recall their theorem, which gives bounds on the sum and the product of the chromatic number of a graph with that of its complement. we also provide a new characterization of the certain graph classes.
\end{abstract}

{\bf Keywords:} Chromatic Number of a graph, Chromatic Index of a graph, Line Graph.

\noindent \textbf{Mathematics Subject Classification 2010: 05C15}.

\section{Introduction}

For all  terms and definitions, not defined specifically in this paper, we refer to \cite{FH}. Unless mentioned otherwise, all graphs considered here are simple, finite and have no isolated vertices.
\\Many problems in extremal graph theory seek the extreme values of graph parameters
on families of graphs. The classic paper of Nordhaus and Gaddum \cite{KCA} study the extreme values of the sum (or product) of a parameter on a graph and its complement, following  solving these problems for the chromatic number on n-vertex graphs. In this paper, we study such problems for some graphs and their associated graphs.

\begin{definition}\rm{
\cite{FH} A \textit{coloring} of a graph is an assignment of colors to its vertices so that no two adjacent vertices have the same color. The set of all vertices with any one color is independent and is called a color class. An \textit{$n$-coloring} of a graph G uses n colors; it thereby partitions V into n color classes. The \textit{chromatic number} $\chi(G)$ is defined as the minimum $n$ for which $G$ has an $n$ - coloring. A graph $G$ is \textit{$n$-colorable} if $\chi(G)\leq n$ and is \textit{$n$-chromatic} if $\chi(G)=n$. }
\end{definition}

\begin{definition}\rm{
\cite{FH} An \textit{edge-coloring} or \textit{line-coloring} of a graph $G$ is an assignment of colors to its edges (lines) so that no two adjacent edges (lines) are assigned the same color. An \textit{n-edge-coloring}  of $G$ is an edge-coloring of $G$ which uses exactly $n$ colors. The \textit{edge-chromatic number} $\chi'(G)$ is the minimum $n$ for which $G$ has an $n$-edge-coloring.}
\end{definition}

Recall the following theorem, which gives bounds on the sum and the product of the chromatic number of a graph with that of its complement.  
\begin{theorem}\rm{
\cite{KCA} If $G$ is a graph with $V(G)=n$ and chromatic number $\chi(G)$ then}
\end{theorem}
\begin{equation}
2\sqrt{n}\leq \chi(G)+\chi(\bar{G})\leq n+1  
\end{equation}
\begin{equation}
n\leq\chi(G).\chi(\bar{G})\leq \frac{(n+1)^2}{4}
\end{equation}
And there is no possible improvement of any of these bounds.  In fact, much more can be said.  Let $n$ be a positive integer.  For every two positive integers $a$ and $b$, 
\begin{equation}
2\sqrt{n}\leq a+b \leq n+1
\end{equation}
\begin{equation}
n \leq ab \leq \frac{(n+1)^2}{4}
\end{equation}
There is a graph $G$ of order $n$ such that $\chi(G)=a$ and $\chi(\bar{G})=b$.
\begin{definition}{\rm
\cite{CGT} The \textit{chromatic index} (or \textit{edge chromatic number}) $\chi'(G)$ of a graph $G$ is the minimum positive integer $k$ for which $G$ is $k-$edge colorable. Furthermore, $\chi'(G)=\chi(L(G))$ for every nonempty graph $G$.}
\end{definition}
\begin{theorem}{\rm
\cite{FH} For any graph $G$, the edge-chromatic number satisfies the inequalities}
\begin{equation}
\Delta\leq\chi'\leq\Delta+1
\end{equation}
\end{theorem}
\begin{theorem}{\rm
\cite{CGT} (Konig's Theorem) If $G$ is a nonempty bipartite graph, then $\chi'(G)=\Delta(G)$.}
\end{theorem}
\begin{theorem}{\rm
\cite{CH} Let $G=K_n$, the complete graph on $n$ vertices, $n\geq2$. Then}
\[\chi'(G)=
\left\{
\begin{array}{ll}
\Delta(G)& \text{if $n$ is even} \\
\Delta(G)+1 & \text{if $n$ is odd}
\end{array}\right.\]
\end{theorem}
We denote the chromatic number of a graph $G$ is denoted by $\chi(G)$ and the complement of G is denoted by $\bar{G}$.
\\This work is motivated by the inspiring talk given by Dr. J Paulraj Joseph, Department of Mathematics, Manonmaniam Sundaranar University, Tirunelveli - 627012, TamilNadu, Titled \textbf{Bounds on sum of graph parameters - A survey}, at the National Conference on Emerging Trends in Graph Connections (NCETGC-2014), on January 8-10, 2014, at the Department of Mathematics, University of Kerala, Kariavattom, Kerala.
\section{New Results}
\begin{definition}{\rm
\cite{DBW} The \textit{line graph} of a graph $G$, written $L(G)$, is the graph whose vertices are the edges of $G$, with $ef\in{E(L(G))}$ when $e=uv$ and $f=vw$ in G.}
\end{definition}
With the above background, we now prove the following.
\begin{proposition}
For a Complete graph $K_n$, $n\geq 2$,
\[\chi(K_n)+\chi(L(K_n))=
\left\{
\begin{array}{ll}
2n-1 & \text{if $n$ is even} \\
2n & \text{if $n$ is odd}
\end{array}\right.\]
\[\chi(K_n).\chi(L(K_n))=
\left\{
\begin{array}{ll}
n(n-1) & \text{if $n$ is even} \\
n^2 & \text{if $n$ is odd}
\end{array}\right.\]
\end{proposition}
\begin{proof}
We know that $\chi(K_n)= n$ for all positive integer $n$.  Let $L(K_n)$ denotes the line graph of $K_n$.  Then, $K_n$ is $(n-1)$ regular and by definition, $\chi(L(K_n))= \chi'(K_n)$.  By theorem \cite{CH},
\[\chi'(K_n)=
\left\{
\begin{array}{ll}
\Delta=n-1 & \text{if $n$ is even} \\
\Delta + 1=n & \text{if $n$ is odd}
\end{array}\right.\]
ie,
\[\chi(L(K_n))= \chi'(K_n)=
\left\{
\begin{array}{ll}
\Delta=n-1 & \text{if $n$ is even} \\
\Delta + 1=n & \text{if $n$ is odd}
\end{array}\right.\]
Therefore
\[\chi(K_n)+\chi(L(K_n))=
\left\{
\begin{array}{ll}
n+n-1=2n-1 & \text{if $n$ is even} \\
n+n=2n & \text{if $n$ is odd}
\end{array}\right.\]
Similarly,
\[\chi(K_n).\chi(L(K_n))=
\left\{
\begin{array}{ll}
n(n-1) & \text{if $n$ is even} \\
n.n=n^2 & \text{if $n$ is odd}
\end{array}\right.\]
\end{proof}
\begin{proposition}
For a Complete bipartite graph $K_{m,n}$, $m,n\geq 0$,
\\
\centerline {$\chi(K_{m,n})+\chi(L(K_{m,n}))= 2+max~(m,n)$  and}
\\
\centerline {$\chi(K_{m,n}).\chi(L(K_{m,n}))= 
2~max~(m,n)$}
\end{proposition}
\begin{proof}
We know that $\chi(K_{m,n})= 2$ for all positive integer $m,n$.  Let $L(K_{m,n})$ denotes the line graph of $K_{m,n}$.  Then, by definition, $\chi(L(K_{m,n}))= \chi'(K_{m,n})$. \\ 
Therefore By theorem \cite{CGT}, $\chi(L(K_{m,n}))= \chi'(K_{m,n})=max~(m,n)$.
\\\\Then $\chi(K_{m,n})+\chi(L(K_{m,n}))= 2+max~(m,n)$ and\\
$\chi(K_{m,n}).\chi(L(K_{m,n}))= 2~max~(m,n)$
\end{proof}
\begin{corollary}
For a star graph $K_{1,n}$, 
\\
\centerline {$\chi(K_{1,n})+\chi(L(K_{1,n}))=n+2$   and}
\\
\centerline {$\chi(K_{1,n}).\chi(L(K_{1,n}))=2n$}
\end{corollary}
\begin{proof}
Since any two edges of a star graph  $K_{1,n}$ are adjacent each other, then its line graph is a complete graph with $n$ vertices.  We know $\chi(K_{1,n})=2$ and $\chi(L(K_{1,n}))=n$.
\\
Therefore $\chi(K_{1,n})+\chi(L(K_{1,n}))= n+2$ and
\\
$\chi(K_{1,n}).\chi(L(K_{1,n}))=2n$.
\end{proof}

A \textit{bistar graph}  $(B_{m,n})$ is a graph obtained by attaching $m$ pendent edges to one end point and $n$ pendent edges to the other end point of $K_2$.
\\
The following result establishes the sum and product of chromatic numbers of a bistar graph and its line graph.
\begin{proposition}
For a bistar graph $B_{m,n}$,
\\
\centerline {$\chi(B_{m,n})+\chi(L(B_{m,n}))= 2+max~(m,n)$  and}
\\
\centerline {$\chi(B_{m,n}).\chi(L(B_{m,n}))= 2~ max~(m,n)$}
\end{proposition}
\begin{proof}
Let $u$, $v$ be two vertices of $K_2$.  Let $m$ edges be attached to $u$ and $n$ edges be attached to $v$.  Since all $m$ edges at $u$ are adjacent to each other and all $n$ edges at $v$ are adjacent to each other, its line graph is the one point union of 2 complete graphs $K_m$ and $K_n$. \\ Then
\[\chi(L(B_{m,n}))=
\left\{
\begin{array}{ll}
m & \text{if $m>n$} \\
n & \text{otherwise}
\end{array}\right.\]
Therefore
\\
\centerline {$\chi(B_{m,n})+\chi(L(B_{m,n}))= 2+max~(m,n)$  and}
\\
\centerline {$\chi(B_{m,n}).\chi(L(B_{m,n}))= 2~  max~(m,n)$}
\end{proof}
\begin{proposition}
Let $G$ be a bipartite graph with a bipartition $(X,Y)$ with $|X|=m$ and $|Y|=n$, then
$4\leq\chi(G)+\chi(L(G)\leq2+max~(m,n)$ and\\
$4\leq\chi(G).\chi(L(G)\leq2~max~(m,n)$
\end{proposition}
\begin{proof}
The minimal connected bipartite graph of $m,n$ vertices will be $P_{m+n-1}$ and that of its line graph is $P_{m+n-2}$. Chromatic number of $G$ and $L(G)$ is 2.\\
\\Therefore $\chi(G)+\chi(L(G))= 4$ and $\chi(G).\chi(L(G))= 4$ then\\
\begin{equation}
4\leq\chi(G)+\chi(L(G))\leq2+max~(m,n)
\end{equation}
\begin{equation}
4\leq\chi(G).\chi(L(G))\leq2~max~(m,n)
\end{equation}
\end{proof}
\begin{definition}\rm{
\cite{FH} For $n\geq4$, a \textit{wheel graph} $W_n$ is defined to be the graph $K_1+ C_{n-1}$.}
\end{definition}
\begin{theorem}
The chromatic index of a wheel graph $W_n$ with $n$ vertices is $n-1$. 
\end{theorem}
\begin{proof}
A wheel graph $W_n$ with $n$ vertices is $K_1+ C_{n-1}$.  Suppose $K_1$ lies inside the circle $C_{n-1}$.  Let $e_1,e_2,e_3,...,e_{n-1}$ be the edges incident with the vertex $K_1$ and we need $n-1$ colors to color this $n-1$ edges.  Let $u_1,u_2,u_3,...,u_{n-1}$ be the end vertices of $e_1,e_2,e_3,...,e_{n-1}$, which form the cycle $C_{n-1}$.  Then, there exists $q_1,q_2,q_3,...,q_{n-1}$ edges incident to $u_1,u_2,u_3,...,u_{n-1}$.  For any edge $q_j$ in the cycle $C_{n-1}$, there exists an edge $e_i$ which is not adjacent to $q_j$.  Therefore $e_i$ and $q_j$ can have the same color.  That is using the same set of $n-1$ colors, we can color the edges $q_1,q_2,q_3,...,q_{n-1}$.  That means we can color the edges of a wheel graph $W_n$ with $n-1$ colors or the chromatic index of $W_n$ is $n-1$.
\end{proof}
\begin{proposition}
For a wheel graph $W_n$ on $n$ vertices and $2(n-1)$ edges, $n\geq 4$,
\[\chi(W_n)+\chi(L(W_n))=
\left\{
\begin{array}{ll}
n+3 & \text{if $n$ is even} \\
n+2 & \text{if $n$ is odd}
\end{array}\right.\]
\[\chi(W_n).\chi(L(W_n))=
\left\{
\begin{array}{ll}
4(n-1) & \text{if $n$ is even} \\
3(n-1) & \text{if $n$ is odd}
\end{array}\right.\]
\end{proposition}
\begin{proof}
We know that 
\[\chi(W_n)=
\left\{
\begin{array}{ll}
4 & \text{if $n$ is even} \\
3 & \text{if $n$ is odd}
\end{array}\right.\] 
for all positive integer $n\geq4 $.  Let $L(W_n)$ denotes the line graph of $W_n$.  Then, 
$\chi(L(W_n))= \chi'(W_n) = (n-1)$
Therefore
\[\chi(W_n)+\chi(L(W_n))=
\left\{
\begin{array}{ll}
4+(n-1)=n+3 & \text{if $n$ is even} \\
3+(n-1)=n+2 & \text{if $n$ is odd}
\end{array}\right.\]
Similarly,
\[\chi(W_n).\chi(L(W_n))=
\left\{
\begin{array}{ll}
4(n-1) & \text{if $n$ is even} \\
3(n-1) & \text{if $n$ is odd}
\end{array}\right.\]
\end{proof}
\begin{definition}\rm{
\cite{JAG} \textit{Helm graphs} are graphs obtained from a wheel by attaching one pendant edge to each vertex of the cycle.}
\end{definition}
\begin{theorem}\rm{
The chromatic index of a helm graph $H_n$ with $2n+1$ vertices and $3n$ edges is $n$.}
\end{theorem}
\begin{proof}
Let $u_1$ is the central vertex and $v_1,v_2,v_3,...,v_n$ be the vertices of the cycle.  Let $w_1,w_2,w_3,...,w_n$ be the pendent vertices attached to $v_1,v_2,v_3,...,v_n$ respectively.  Let $e_1,e_2,e_3,...,e_n$ be the edges incident on the vertex $u_1$.  Let $l_1,l_2,l_3,...,l_n$ be the edges of the cycle formed by the vertices $v_1,v_2,v_3,...,v_n$.  Let $q_1,q_2,q_3,...,q_n$ be the pendent edges.  Since each $e_1,e_2,e_3,...,e_n$ are adjacent to each other, to color the edges $e_1,e_2,e_3,...,e_n$, we need atleast $n$ colors.  For every edge $l_i$, we can find atleast one edge $e_j$ such that $l_i$and$e_j$ are non-adjacent.  Color the edge $l_i$ with the same color of $e_j$.  Using the same set of $n$ colors, we can color all the edges $e_1,e_2,e_3,...,e_n$ and $l_1,l_2,l_3,...,l_n$.  For any edge $q_k$, there will be atleast one edge $e_j$ and alteast one edge $l_i$ with the same color and are non-adjacent to $q_k$.  Now assign this color to $q_k$.  Hence we color all the vertices of helm graph using the same set of $n$ colors.
\end{proof}

\begin{proposition}
For a helm graph $H_n$ on $2n+1$ vertices, and $3n$ edges, $n\geq 3$,
\[\chi(H_n)+\chi(L(H_n))=
\left\{
\begin{array}{ll}
n+4 & \text{if $n$ is even} \\
n+3 & \text{if $n$ is odd}
\end{array}\right.\]
\[\chi(H_n).\chi(L(H_n))=
\left\{
\begin{array}{ll}
4n & \text{if $n$ is even} \\
3n & \text{if $n$ is odd}
\end{array}\right.\]
\end{proposition}
\begin{proof}
We know that 
\[\chi(H_n)=
\left\{
\begin{array}{ll}
4 & \text{if $n$ is even} \\
3 & \text{if $n$ is odd}
\end{array}\right.\] 
for all positive integer $n\geq4 $.  Let $L(H_n)$ denotes the line graph of $H_n$.  Then, 
$\chi(L(H_n))= \chi'(H_n) = n$
\\
Therefore
\[\chi(H_n)+\chi(L(H_n))=
\left\{
\begin{array}{ll}
n+4 & \text{if $n$ is even} \\
n+3 & \text{if $n$ is odd}
\end{array}\right.\]
Similarly,
\[\chi(H_n).\chi(L(H_n))=
\left\{
\begin{array}{ll}
4n & \text{if $n$ is even} \\
3n & \text{if $n$ is odd}
\end{array}\right.\]
\end{proof}
\begin{definition}\rm{
\cite{DBW} Given a vertex $x$ and a set $U$ of vertices, an $x$, $U-$fan is a set of paths from $x$ to $U$ such that any two of them share only the vertex $x$.}
\end{definition}
\begin{theorem}
The chromatic index of a fan graph $F_{1,n}$ with $n+1$ vertices is $n$. 
\end{theorem}
\begin{proof}
The fan graph $F_{1,n}$ with $n+1$ vertices is $K_1+ P_{n-1}$.  Let $e_1,e_2,e_3,...,e_n$ be the edges incident with the vertex $K_1$ and we need $n$ colors to color this $n$ edges.  Let $q_1,q_2,q_3,...,q_{n-1}$ be the edges in the path $P_{n-1}$.  For any edge $q_j$ in the path $P_{n-1}$, there exists an edge $e_i$ which is not adjacent to $q_j$.  Therefore $e_i$ and $q_j$ can have the same color.  That is, by taking $(n-1)$ colors out of $n$ colors, we can color the edges $q_1,q_2,q_3,...,q_{n-1}$.  That is we can color the edges of a fan graph $F_{1,n}$ with $n$ colors or the chromatic index of $F_{1,n}$ is $n$.
\end{proof}
\begin{proposition}
For a fan graph $F_{1,n}$, 
\\
\centerline {$\chi(F_{1,n})+\chi(L(F_{1,n}))=n+4$   and}
\\
\centerline {$\chi(F_{1,n}).\chi(L(F_{1,n}))=3(n+1)$}
\end{proposition}
\begin{proof}
For a fan graph $F_{1,n}$, with $e\geq1$, we have $\chi(F_{1,n})=3$ for all positive integer $n\geq2$.  Let $L(F_{1,n})$ denotes the line graph of $F_{1,n}$.  Then $\chi(L(F_{1,n}))=\chi'(F_{1,n})=n$.
\\
Therefore $\chi(F_{1,n})+\chi(L(F_{1,n}))=3+n=n+3$ and
\\
$\chi(F_{1,n}).\chi(L(F_{1,n}))=3n$.
\end{proof}
\section{Conclusion}
The theoretical and experimental results obtained in this research may provide a better insight into the problems involving chromatic number by improving the known lower and upper bounds on sums and products of chromatic numbers of a graph $G$ and an associated graph of $G$.  More properties and characteristics of operations on chromatic number and also other graph parameters are yet to be investigated.  The problems of establishing the inequalities on sums and products of chromatic numbers for various graphs and graph classes still remain unsettled. All these facts highlight a wide scope for further studies in this area.

\end{document}